\documentclass[a4paper,12pt]{article}
\usepackage{amsmath,amsthm,amsfonts}
\usepackage[a4paper,centering]{geometry}
\usepackage[pdfdisplaydoctitle,colorlinks,urlcolor=blue,linkcolor=blue,citecolor=blue]{hyperref}
\usepackage{mathtools}
\newtheorem{theorem}{Theorem}[section]
\newtheorem{lemma}[theorem]{Lemma}
\newtheorem{proposition}[theorem]{Proposition}
\theoremstyle{definition}
\newtheorem{definition}[theorem]{Definition}
\theoremstyle{remark}
\newtheorem{remark}[theorem]{Remark}
\numberwithin{equation}{section}
\hypersetup{%
  pdftitle={Global regularity for a slightly supercritical hyperdissipative Navier--Stokes system},
  pdfauthor={D. Barbato, F. Morandin, M. Romito}
}
\newcommand{\Div}{\nabla\cdot}
\newcommand{\N}{\mathbb{N}}
\newcommand{\Z}{\mathbb{Z}}
\newcommand{\R}{\mathbb{R}}
\newcommand{\C}{\mathbb{C}}
\DeclareMathOperator{\re}{Re}
\DeclareMathOperator{\im}{Im}
\newcommand{\scalar}[1]{\langle #1\rangle}
\newcommand\zds{\Z^d\setminus\{0\}}
\newcommand\natu{\N_0}
\newcommand{\vep}[1][\epsilon]{v_{#1}}
\def\MRnum#1 #2\empty{#1}
\newcommand{\MRhref}[2]{%
  \href{http://www.ams.org/mathscinet-getitem?mr=#1}{#2}
}
\newcommand{\MR}[1]{%
  \relax\ifhmode\unskip\space\fi
  \MRhref{\MRnum#1\empty}{\texttt{\tiny[MR\MRnum#1\empty]}}
}
\newcommand{\arxiv}[1]{\href{http://arxiv.org/abs/#1}{arXiv:#1}}
\title{Global regularity for a slightly supercritical hyperdissipative Navier--Stokes system}
\author{
  D.~Barbato\\
    \small Universit\`a di Padova\\[-.2em]
    \scriptsize\texttt{barbato@math.unipd.it}
  \and
  F.~Morandin\\
    \small Universit\`a di Parma\\[-.2em]
    \scriptsize\texttt{francesco.morandin@unipr.it}
  \and
  M.~Romito\\
    \small Universit\`a di Pisa\\[-.2em]
    \scriptsize\texttt{romito@dm.unipi.it}
}
\begin{document}
\maketitle
\begin{abstract}
  We prove global existence of smooth solutions
  for a slightly supercritical hyperdissipative
  Navier--Stokes under the optimal condition
  on the correction to the dissipation.
  This proves a conjecture formulated by
  Tao \cite{Tao2009}.
\end{abstract}
\section{Introduction}

Let $d\geq 3$
and consider the generalized Navier--Stokes system
\begin{equation}\label{e:gnse}
  \begin{cases}
    \frac{\partial u}{\partial t} + (u\cdot\nabla)u + \nabla p + D_0^2 u
      = 0,\\
    \Div u
      = 0,\\
    \int_{[0,2\pi]^d} u(t,x)\,dx
      = 0,
  \end{cases}
\end{equation}
on $[0,2\pi]^d$ with periodic boundary conditions,
where $D_0$ is a Fourier multiplier with non--negative symbol
$m$. The Navier--Stokes system is recovered when
$m(k) = |k|$. If
\begin{equation}\label{e:asstao1}
  m(k)
    \geq c\frac{|k|^{\frac{d+2}4}}{G(|k|)},
\end{equation}
where $G:[0,\infty)\to[0,\infty)$ is a non--decreasing
function such that
\begin{gather}
  \int_1^\infty\frac{ds}{s G(s)^4}
    =\infty,
    \label{e:asstao2}\\
  \frac{G(x)}{|x|^{\frac{d+2}4}}
    \text{ eventually non--increasing},
    \label{e:asstao3}
\end{gather}
then in \cite{Tao2009} it is proved\footnote{The proof of the result
    of \cite{Tao2009} is
    given in $\R^d$, but it can be easily extended to the
    periodic setting, see \cite[Remark 2.1]{Tao2009}.}
that \eqref{e:gnse} has a global smooth solution for every
smooth initial condition. The result has been later extended
to the two dimensional case by \cite{KatTap2012}.

A heuristic argument developed in \cite{Tao2009} and
based on the comparison between the speed of
propagation of a (possible) blow--up and the rate
of dissipation suggests that regularity should still
hold under the weaker condition
\begin{equation}\label{e:assbmr}
  \int_1^\infty\frac{ds}{s G(s)^2}
    =\infty.
\end{equation}
The main result of this paper, contained in the
following theorem, is a complete proof of this
conjecture.
\begin{theorem}\label{t:main1}
  Let $d\geq 2$ and assume \eqref{e:asstao1}, \eqref{e:asstao3}
  and \eqref{e:assbmr} for a non--decreasing function
  $G:[0,\infty)\to[0,\infty)$. Then \eqref{e:gnse} has a
  global smooth solution for every smooth initial condition.
\end{theorem}
A simple version of this conjecture, when reformulated on
a toy model, has been proved for the dyadic model in
\cite{BarMorRom2014}.
Actually, for that model one could prove regularity in full
supercritical regime, with $m(k)=|k|$, as was done
in~\cite{BarMorRom2011}, but it was natural to develop there some of
the main ideas on which also this paper is based.
In fact here we prove that the equations for the velocity can be
reduced to a suitable dyadic--like model, with infinitely many
interactions though. A more sophisticated version of the arguments
of \cite{BarMorRom2014} ensures regularity of this dyadic model and,
in turns, of the solution of the problem \eqref{e:gnse} above.

Our technique for proving Theorem~\ref{t:main1} is flexible
enough to include an additional critical parameter.
Consider the following generalized Leray $\alpha$--model,
\begin{equation}\label{e:glanse}
  \begin{cases}
    \frac{\partial v}{\partial t} + (u\cdot\nabla)v + \nabla p + D_1 v
      = 0,\\
    v
      = D_2 u,\\
    \Div v
      = 0,\\
    \int_{[0,2\pi]^d} v(t,x)\,dx
      = \int_{[0,2\pi]^d} u(t,x)\,dx
      = 0,
  \end{cases}
\end{equation}
where $D_1$ and $D_2$ are Fourier multipliers with
non--negative symbols $m_1$ and $m_2$.
\begin{theorem}\label{t:main2}
  Let $d\geq2$, $\alpha,\beta\geq0$, and assume
  \[
    m_1(k)
      \geq c\frac{|k|^\alpha}{g(|k|)},
        \qquad
    m_2(k)
      \geq c|k|^\beta,
        \qquad
    \alpha+\beta\geq\frac{d+2}2,
  \]
  where $g:[0,\infty)\to[0,\infty)$ is
  a non--decreasing function such that $x^{-\alpha}g(x)$
  is eventually non--increasing, and
  \begin{equation}\label{e:assbmrtrue}
    \int_1^\infty\frac{ds}{s g(s)}
      = \infty.
  \end{equation}
  Then \eqref{e:glanse} has a global smooth solution
  for every smooth initial condition.
\end{theorem}
Under the assumptions of Theorem~\ref{t:main1}, if
$\beta=0$, $\alpha=\frac{d+2}{2}$, $g(x) = G(x)^2$,
$m_2(k)=1$, and $m_1(k)=m(k)^2$, then the assumptions
of Theorem~\ref{t:main2} are met. Therefore 
Theorem~\ref{t:main1} follows immediately from
Theorem~\ref{t:main2}, and it is sufficient to prove only
the second result.

Our results hold as well when the problems are considered
in $\R^d$, since in our method large scales play no
significant role (see Remark~\ref{r:erredi}).

The model~\eqref{e:glanse} with $g\equiv1$ was introduced by Olson and
Titi in~\cite{OlsTit2007}. They proposed the idea that a weaker
non-linearity and a stronger viscous dissipation could work together
to yield regularity. Their statement uses though a stronger hypothesis
$\alpha+\frac{\beta}{2}\geq\frac{d+2}{2}$ and this result was later
logarithmically improved in~\cite{Yam2012} with condition~\eqref{e:asstao2}.

Our results are also relevant in view of the analysis
in \cite{Tao2014} (see Remark 5.2 therein), since they confirm
that the condition \eqref{e:assbmrtrue} is optimal,
when general non--linear terms with the same
scaling are considered.

The proof of the above theorem is based on two crucial
ideas. The first idea is that smoothness of \eqref{e:glanse}
can be reduced to the smoothness of a suitable shell model,
obtained by averaging the energy of a solution
of \eqref{e:glanse} over dyadic shells in Fourier space.
We believe that this reduction may be interesting beyond
the scope of this paper. The second idea is that the
overall contribution of energy and dissipation over
large shells satisfies a recursive inequality.
Under condition \eqref{e:assbmrtrue} dissipation
dumps significantly the flow of energy towards small
scales and ensures smoothness.
This is a more sophisticated
version of the result obtained in \cite{BarMorRom2014},
due to the larger number of interactions between shells.

The paper is organized as follows. In Section~\ref{s:reduction}
we derive the \emph{shell approximation} of a solution
of \eqref{e:glanse}. The recursive formula is obtained
in Section~\ref{s:recursion}. In Section~\ref{s:cascade}
we deduce exponential decays of shell modes by the
recursive formula. The appendix~\ref{s:localexuniq}
contains, for the sake of completeness, a standard existence
and uniqueness result.
\section{From the generalized Fourier Navier--Stokes to the dyadic equation}
  \label{s:reduction}

This section contains one of the crucial steps in our approach.
We show that the proof of Theorem~\ref{t:main2} can be reduced
to a proof of decay of solutions of a suitable shell model.
For simplicity and without loss of generality from now on we
assume that
\[
  m_1(k)
    = \frac{|k|^\alpha}{g(|k|)},
      \qquad
  m_2(k)
    \geq|k|^\beta.
\]
\subsection{The shell approximation}

The dynamics of our generalized version of Navier-Stokes equation in
Fourier decomposition reads
\begin{equation}\label{e:fourier_ns}
  \begin{cases}
    \displaystyle
    v_k'=-\frac{|k|^\alpha}{g(|k|)}v_k-i\sum_{h\in\zds}\frac{\scalar{v_h,k}}{|h|^\beta}P_k(v_{k-h}),\\
    \scalar{v_k,k}=0,\\
    v_{-k}=\overline{v_k},
  \end{cases}
  \qquad k\in\zds
\end{equation}
where $P_k(w):=w-\frac{\scalar{w,k}}{|k|^2}k$ and $v_0 = 0$.
A solution is a family $(v_k)_{k\in\zds}$ where each
$v_k=v_k(t)$ is a differentiable map from $[0,\infty)$ to $\C^d$
satisfying~\eqref{e:fourier_ns} for all times.

As is common in Littlewood-Paley theory, let $\Phi:[0,\infty)\to[0,1]$
be a smooth function such that $\Phi\equiv1$ on $[0,1]$,
$\Phi\equiv0$ on $[2,\infty)$ and $\Phi$ is strictly decreasing on
$[1,2]$.
For $x\geq0$, let $\psi(x):=\Phi(x)-\Phi(2x)$, so that $\psi$
is a smooth bump function supported on $(\frac12,2)$ satisfying
\[
\sum_{n=0}^\infty\psi(x/2^n)
=1-\Phi(2x)
\equiv1
,\qquad x\geq1.
\]
Notice that it is elementary to show that $\sqrt{\psi}$
is Lipschitz continuous.

Let $\natu$ denote the set of non-negative integers. For all
$n\in\natu$ we introduce the radial maps $\psi_n:\R^d\to[0,1]$ defined
by $\psi_n(x)=\psi(2^{-n}|x|)$. Notice that
\[
\sum_{n\in\natu}\psi_n(x)
\equiv1
,\qquad x\in\zds
\]
In Littlewood-Paley theory one typically defines $\psi_n$ for all
$n\in\Z$, introduces objects like
\[
P_n(x):=\sum_{k\in\Z^d}\psi_n(k)v_ke^{i\scalar{k,x}}
\]
and then proves that $u=\sum_nP_n$. Since these $P_n$ are not
orthogonal\footnote{They are in fact \emph{almost orthogonal} in the
  sense that $\scalar{P_n,P_m}_{L^2}=0$ whenever $|m-n|\geq2$.}  this
does not give a nice decomposition of energy, as
\[
\sum_{n\in\Z}\|P_n\|_{L^2}^2
\neq\sum_{k\in\Z^d}|v_k|^2
=\|u\|_{L^2}^2.
\]

Thus instead of $P_n(x)$ we introduce a sort of \emph{square-averaged}
Littlewood-Paley decomposition. Let
\begin{equation}\label{e:xn_def}
X_n(t)
:=\left(\sum_{k\in\Z^d}\psi_n(k)|v_k(t)|^2\right)^{1/2}
,\qquad n\in\natu,\quad t\geq0
\end{equation}
Then clearly
\[
\sum_{n\in\natu}X_n^2
=\sum_{k\in\Z^d}|v_k|^2
=\|u\|_{L^2}^2.
\]
\begin{remark}
One major difference with respect to the usual Littlewood-Paley theory
is that it is impossible to recover $v$ from these $X_n$ (as it was
with the components $P_n(x)$), since they are averaged both in
the physical space and over one shell of the frequency space.
\end{remark}

We will denote by $H^\gamma$ the Hilbert--Sobolev space of
periodic functions with differentiation index $\gamma$,
namely
\begin{equation}\label{e:sobolev}
  H^\gamma
    = \{v = (v_k)_{k\in\Z^d}: \sum (1+|k|^2)^\gamma|v_k|^2<\infty\}.
\end{equation}
\begin{definition}
If~\eqref{e:xn_def} holds, we say that $X=(X_n(t))_{n\in\natu,t\geq0}$
is the \emph{shell approximation} of $v$.
\end{definition}

If $v\in H^\gamma$ and $X$ is its shell approximation, then
\[
  \sum_n 2^{2\gamma n}X_n^2
    = \sum_k\Bigl(\sum_n 2^{2\gamma n}\psi_n(k)\Bigr)|v_k|^2
    \approx\sum_k |k|^{2\gamma}|v_k|^2
    = \|v\|_{H^\gamma}^2.
\]
Hence, $v(t)\in C^\infty$ if and only if $\sup_n 2^{\gamma n}X_n<\infty$
for every $\gamma>0$. In view of Theorem~\ref{t:localexuniq}
(see page \pageref{t:localexuniq}), Theorem~\ref{t:main2} follows
if we can prove the following result.
\begin{theorem}\label{t:main2bis}
  Under the same assumptions of Theorem~\ref{t:main2}, let
  $v(0)$ be smooth and periodic, and let $m\geq2+\frac{d}2$.
  If $v$ is a solution of \eqref{e:glanse} in $H^m$ on its maximal
  interval of existence $[0,T_\star)$, $X$ is its shell approximation
  and
  \[
    \sup_{[0,T_\star)}\sum 2^{2mn}X_n^2
      <\infty,
  \]
  then $T_\star=\infty$.
\end{theorem}
\subsection{The shell solution}

We want to write a system of equation for the shell approximation
of a solution of \eqref{e:glanse}. We give a more formal connection
between \eqref{e:glanse} and its shell equation because we believe
the notion will result useful beyond the scopes of the present work.

Define the set $I$ as follows,
\begin{equation}\label{e:def_I}
I:=\left\{(l,m,n)\in\natu^3:\begin{array}l
\text{the difference between the two largest}\\
\text{integers among $l$, $m$ and $n$ is at most 2}
\end{array}
\right\}.
\end{equation}
We are now ready to introduce the shell model ODE for the energy of
each shell, equation~\eqref{e:shell_ns}.

\begin{definition}[shell solution]\label{d:shell_solution}
Let $X=(X_n)_{n\in\natu}$ be a sequence of real valued maps,
$X_n:[0,\infty)\to\R$.  We say that $X$ is a \emph{shell solution} if
  there are two families of real valued maps $\chi=(\chi_n)_{n\in\natu}$
  and $\phi=(\phi_{(l,m,n)})_{(l,m,n)\in I}$ such that
\begin{equation}\label{e:shell_ns}
\frac d{dt}X_n^2(t)
=-\chi_n(t)X_n^2(t)+\sum_{\substack{l,m\in\natu\\(l,m,n)\in I}}\phi_{(l,m,n)}(t)X_l(t)X_m(t)X_n(t),
\end{equation}
for all $n\in\natu$ and $t>0$ where the sum above is understood as absolutely
convergent, and $\eta,\phi$ satisfies the following properties,
\begin{enumerate}
\item the family $\phi$ is \emph{antisymmetric}, in the sense that
\[
\phi_{(l,m,n)}(t)=-\phi_{(l,n,m)}(t),
\qquad (l,m,n)\in I,\ t\geq0,
\]
\item there exist two positive constants $c_1$ and $c_2$ for which,
\begin{equation}\label{e:bound_chi_and_phi_def_shell_sol}
\chi_n(t)\geq c_1\frac{2^{\alpha n}}{g(2^{n+1})}
\qquad\text{and}\qquad
\left|\phi_{(l,m,n)}(t)\right|\leq c_22^{(\frac d2+1-\beta)\min\{l,m,n\}}
\end{equation}
for all $(l,m,n)\in I$ and $t\geq0$.
\end{enumerate}
\end{definition}

\begin{remark}
We will prove below that the shell approximation of a solution
of \eqref{e:glanse} is a shell solution. It is easy to check
that the dissipation term is local as expected, due to the
way the shell components of a solution interact in
the model's dynamics. As for the nonlinear term, it turns out
that the set $I$ of the triples of indices $(l,m,n)$ for which there
may be interaction between the shell components $l$, $m$ and $n$ is
quite small. This is basically because in the Fourier space, three
components may interact only if they are the sides of a triangle and
by triangle inequality their lengths cannot be in three shells far
away from each other.
\end{remark}

\begin{remark}
  To ensure that the sum in \eqref{e:shell_ns} is absolutely convergent,
  it is sufficient to assume that the sequence $(X_n(t))_{n\in\natu}$
  is square summable (this will be a consequence of the energy inequality,
  see Definition~\ref{d:energyineq}). Indeed, if $n$ is not the smallest
  index, then the sum is extended to a finite number if
  indices. Otherwise, $\phi_{(l,m,n)}$ is constant
  with respect to $l,m$.
\end{remark}

\begin{remark}\label{r:cancellations_phi}
The antisymmetric property is what makes the non--linearity
of \eqref{e:shell_ns} \emph{formally} conservative. In fact using
antisymmetry, a change of
variable ($m'=n$ and $n'=m$) and the fact that $(l,m',n')\in I$ if
and only if $(l,n',m')\in I$, one could formally write,
\begin{multline*}
-\sum_{\substack{l,m,n\in\natu\\(l,m,n)\in I}}\phi_{(l,m,n)}X_lX_mX_n
=\sum_{\substack{l,m,n\in\natu\\(l,m,n)\in I}}\phi_{(l,n,m)}X_lX_mX_n\\
=\sum_{\substack{l,m',n'\in\natu\\(l,n',m')\in I}}\phi_{(l,m',n')}X_lX_{m'}X_{n'}
=\sum_{\substack{l,m',n'\in\natu\\(l,m',n')\in I}}\phi_{(l,m',n')}X_lX_{m'}X_{n'}
\end{multline*}
If these sums are absolutely convergent, this would prove indeed that the
expression itself is equal to zero.

Since these are infinite sums, these computations are not rigorous
unless we know, for instance, that $\sum_n 2^{2\gamma n}X_n^2<\infty$,
with $\gamma\geq\frac13(\frac{d}2+1-\beta)$, as it can be verified by
an elementary computation.
\end{remark}
\subsection{The shell model as a shell approximation}

The bounds on the coefficients given in
Definition~\ref{d:shell_solution} are in the correct direction to
prove regularity results (and hence Theorem~\ref{t:main2bis}).
The following theorem, which is the main
result of this section shows that they capture the natural scaling of
the shell interactions for the \emph{physical} solutions.

\begin{theorem}\label{t:shell_ode}
If $v$ is a solution of \eqref{e:glanse} on $[0,T]$
and $X$ is its shell approximation, then $X$ is a shell solution.
\end{theorem}

\begin{remark}\label{r:erredi}
  At this stage it is easy to realize that our main results hold
  also in $\R^d$ with minimal changes. Indeed when passing to
  the shell approximation, all large frequencies are considered
  together in the first element of the shell model.
\end{remark}

The proof of Theorem~\ref{t:shell_ode} can be found at the end of this
section. It is based on
Propositions~\ref{p:bound_chi}-\ref{p:bound_phi} below, which give the
actual definitions of $\chi$ and $\phi$ and prove their properties.

\begin{proposition}\label{p:bound_chi}
Let $X$ be the shell approximation of a solution $v$. Define
$\chi_n(t)$ for all $n\in\natu$ and $t\geq0$ as follows
\begin{equation}\label{e:chi_def}
\chi_n(t)
:=\begin{cases}\displaystyle
\frac2{X_n^2(t)}\sum_{k\in\zds}\psi_n(k)\frac{|k|^\alpha}{g(|k|)}|v_k(t)|^2,
&\quad\text{if }X_n(t)\neq0\\[4ex]
\displaystyle
\frac{2^{\alpha n-\alpha+1}}{g(2^{n+1})},
&\quad\text{if }X_n(t)=0
\end{cases}
\end{equation}
Then
\[
\chi_n(t)
\geq\frac{2^{\alpha n-\alpha+1}}{g(2^{n+1})}
,\qquad n\in\natu,t\geq0
\]
\end{proposition}

\begin{proof}
Fix $n\in\natu$ and $t\geq0$. The map $\psi_n$ is supported on
$\{x\in\Z^d:2^{n-1}<|x|<2^{n+1}\}$ and $g$ is non-decreasing, so
\[
\sum_{k\in\zds}\psi_n(k)\frac{|k|^\alpha}{g(|k|)}\left|v_k(t)\right|^2
\geq\sum_{k\in\zds}\psi_n(k)\frac{2^{(n-1)\alpha}}{g(2^{n+1})}\left|v_k(t)\right|^2
=\frac{2^{(n-1)\alpha}}{g(2^{n+1})}X_n^2(t)
\]
where we used~\eqref{e:xn_def}. By~\eqref{e:chi_def} we get the thesis.
\end{proof}
We finally turn our attention to the antisymmetry property and
an upper bound for $\phi_{(l,m,n)}(t)$. The statement is as follows.

\begin{proposition}\label{p:bound_phi}
Let $X$ be the shell approximation of a solution $v$. Define
$\phi_{(l,m,n)}(t)$ for all $l,m,n\in\natu$ and $t\geq0$ as
\begin{multline}\label{e:phi_def}
\phi_{(l,m,n)}(t)
:=\frac2{X_l(t)X_m(t)X_n(t)}\cdot\\
  \cdot\sum_{\substack{h,k\in\Z^d\\h\neq0}}\psi_l(h)\psi_m(k-h)\psi_n(k)\frac{\im\{\scalar{v_h(t),k}\scalar{v_{k-h}(t),v_k(t)}\}}{|h|^\beta},
\end{multline}
(unless $X_l(t)X_m(t)X_n(t)=0$, in which case $\phi_{(l,m,n)}(t):=0$).

Then:
\begin{enumerate}
\item $\phi_{(l,m,n)}(t)=0$ for all $(l,m,n)\notin I$ and all $t\geq0$.
\item $\phi_{(l,m,n)}(t)=-\phi_{(l,n,m)}(t)$ for all $l,m,n\in\natu$ and all $t\geq0$.
\item For any $\beta\geq0$ there exists a constant $c_3>0$
  depending only on $d$, $\beta$ and $\psi$ such that
\begin{equation}\label{e:bound_phi}
|\phi_{(l,m,n)}(t)|
\leq c_3 2^{(\frac d2+1-\beta)\min\{l,m,n\}} 
,\qquad (l,m,n)\in I,t\geq0.
\end{equation}
\end{enumerate}
\end{proposition}

For the proof we need a couple of lemmas.
\begin{lemma}\label{l:sign_change_sum_k_of_phi}
Suppose $v=(v_k)_{k\in\Z^d}$ is a complex field over $\Z^d$ such that
for all $k\in\Z^d$, $\scalar{k,v_k}=0$ and
$\overline{v_k}=v_{-k}$. Then for all $h\in\Z^d$,
\begin{multline*}
\sum_{k\in\Z^d}\psi_m(k-h)\psi_n(k)\im\{\scalar{v_h,k}\scalar{v_{k-h},v_k}\}\\
=-\sum_{k\in\Z^d}\psi_m(k)\psi_n(k-h)\im\{\scalar{v_h,k}\scalar{v_{k-h},v_k}\}.
\end{multline*}
\end{lemma}

\begin{proof}
Consider the left--hand side. By performing the change of variable $k'=h-k$ we
obtain
\begin{gather*}
\psi_m(k-h)=\psi_m(-k')=\psi_m(k'),\\
\psi_n(k)=\psi_n(h-k')=\psi_n(k'-h),\\
\scalar{v_h,k}=\scalar{v_h,h-k'}=-\scalar{v_h,k'},\\
\scalar{v_{k-h},v_k}=\scalar{v_{-k'},v_{h-k'}}=\scalar{\overline{v_{k'}},\overline{v_{k'-h}}}=\scalar{v_{k'-h},v_{k'}}.
\end{gather*}
The sum for $k\in\Z^d$ is equivalent to the sum for $k'\in\Z^d$ and
this concludes the proof.
\end{proof}

\begin{lemma}\label{l:cs_plus_cardinality_inequality}
Let $v$ be a solution and $X$ its shell approximation.
Then for all $a,b,c\in\natu$ and for all $t\geq0$,
\begin{multline*}
\sum_{h\in\Z^d}\psi_a(h)|v_h(t)|\sum_{k\in\Z^d}\sqrt{\psi_b(k)\psi_c(k-h)}|v_k(t)||v_{k-h}(t)|\leq\\
\leq2^{\frac d2(a+3)}X_a(t)X_b(t)X_c(t).
\end{multline*}
\end{lemma}

\begin{proof}
By Cauchy-Schwarz inequality and formula~\eqref{e:xn_def} we have
that for all $h\in\Z^d$,
\[
\sum_{k\in\Z^d}\sqrt{\psi_b(k)\psi_c(k-h)}|v_k(t)||v_{k-h}(t)|
\leq X_b(t)X_c(t).
\]
Then, let $S_a$ denote the intersection between $\Z^d$ and the support
of $\psi_a$. By inscribing $S_a$ in a cube we can bound its
cardinality with $|S_a|\leq(2^{a+2}+1)^d\leq2^{(a+3)d}$, so
\[
\sum_{k\in\Z^d}\psi_a(k)|v_k(t)|
\leq\left(|S_a|\sum_{k\in S_a}\psi_a^2(k)v_k^2(t)\right)^{1/2}
\leq\left(2^{(a+3)d}\right)^{1/2}X_a(t),
\]
where we used the fact that $\psi_a(k)\leq1$.
\end{proof}

\begin{proof}[Proof of Proposition~\ref{p:bound_phi}]
Consider the definition of $\phi_{(l,m,n)}$,
equation~\eqref{e:phi_def}.  By applying
Lemma~\ref{l:sign_change_sum_k_of_phi}, for fixed $t$, we immediately
conclude that
\[
\phi_{(l,n,m)}
=-\phi_{(l,m,n)}
\qquad l,m,n\in\natu,
\]
and in particular that $\phi_{(l,m,m)}=0$.

Moreover, for all choices of $h$ and $k$, the arguments of $\psi_l$,
$\psi_m$ and $\psi_n$ are the sides of a triangle in $\R^d$, so by the
triangle inequality the size of the largest (wlog $k$) is at most
twice the size of the second largest (wlog $h$). On the other hand for
all $j\in\natu$ the support of $\psi_j$ is
$\{x\in\R^d:2^{j-1}<|x|<2^{j+1}\}$. Thus whenever
$\psi_l(h)\psi_n(k)\neq0$, necessarily $n\leq l+2$ since
\[
2^{n-1}<|k|\leq2|h|<2^{l+2}.
\]
This proves that $\phi_{(l,m,n)}=0$ outside $I$ as defined in
equation~\eqref{e:def_I}.

Finally we prove inequality~\eqref{e:bound_phi} for $(l,m,n)\in I$ with $m<n$. We will
consider separately the two cases $n-m>2$ and $n-m\in\{1,2\}$,
starting from the former.

\bigskip\noindent\textbf{Case 1.} Since $m<n-2$ and $(l,m,n)\in I$,
then $m=\min\{l,m,n\}$ and $|l-n|\leq2$. This means in particular that
for all the non-zero terms of the sum in equation~\eqref{e:phi_def},
tipically $|k-h|<|k|$, so it is convenient to substitute
$\scalar{v_h,k}=\scalar{v_h,k-h}$ in the equation to obtain the
following bound
\[
|\phi_{(l,m,n)}|
\leq\frac2{X_lX_mX_n}\sum_{\substack{h,k\in\Z^d\\h\neq0}}\psi_l(h)\psi_m(k-h)\psi_n(k)\frac{|v_h||k-h||v_{k-h}||v_k|}{|h|^\beta}.
\]
By the definition of $\psi_l$, either $\psi_l(h)=0$ or $|h|\geq
2^{l-1}\geq2^m$. Applying this and the change of variable $k'=k-h$ one
gets,
\[
|\phi_{(l,m,n)}|
\leq\frac{2^{1-\beta m}}{X_lX_mX_n}\sum_{k'\in\Z^d}\psi_m(k')|k'||v_{k'}|\sum_{h\in\Z^d}\psi_l(h)\psi_n(k'+h)|v_h||v_{k'+h}|.
\]
In the same way we can substitute $|k'|\leq2^{m+1}$ and apply
Lemma~\ref{l:cs_plus_cardinality_inequality} (recall that $\psi\leq1$,
so $\psi\leq\sqrt\psi$) to get
\[
|\phi_{(l,m,n)}|
\leq2^{1-\beta m+m+1+\frac d2(m+3)}.
\]
Since in the present case $\min\{l,m,n\}=m$, this proves
inequality~\eqref{e:bound_phi} with $c_3=2^{2+3d/2}$.

\bigskip\noindent\textbf{Case 2.} Suppose now that $n-m\in\{1,2\}$ and
$(l,m,n)\in I$, then $l\leq n+2$ and $\min\{l,m,n\}\geq l-4$. In this
case it is $l$ that can be small with respect to $m$ and $n$, so we
take the terms in $l$ and $h$ outside the internal sum,
\[
|\phi_{(l,m,n)}|
\leq
\frac{2}{X_lX_mX_n}\sum_{h\in\Z^d\setminus\{0\}}\frac{\psi_l(h)}{|h|^\beta}\left|\sum_{k\in\Z^d}\psi_m(k-h)\psi_n(k)\im\{\scalar{v_h,k}\scalar{v_{k-h},v_{k}}\}\right|.
\]
The idea is to exploit the cancellations in the sum over $k$ that
happen when $k-h$ and $k$ are switched. By
Lemma~\ref{l:sign_change_sum_k_of_phi} and the bound $|k|\leq2^{n+1}$
for $k$ in the support of $\psi_m$ or $\psi_n$,
\begin{multline*}
|\phi_{(l,m,n)}|
\leq\frac2{X_lX_mX_n}\sum_{h\in\Z^d\setminus\{0\}}\frac{\psi_l(h)}{|h|^\beta}\\
\cdot\frac12\left|\sum_{k\in\Z^d}(\psi_m(k-h)\psi_n(k)-\psi_m(k)\psi_n(k-h))\im\{\scalar{v_h,k}\scalar{v_{k-h},v_{k}}\}\right|\\
\leq\frac{2^{n+1}}{X_lX_mX_n}\sum_{h\in\Z^d\setminus\{0\}}\frac{\psi_l(h)|v_h|}{|h|^\beta}\sum_{k\in\Z^d}\bigl|\psi_m(k-h)\psi_n(k)-\psi_m(k)\psi_n(k-h)\bigr| \left|v_{k-h}\right| \left|v_{k}\right|.
\end{multline*}
We turn our attention to the term
$\psi_m(k-h)\psi_n(k)-\psi_m(k)\psi_n(k-h)$ and show that it is
small. Let $L$ denote the Lipschitz constant of the function
$\psi^{1/2}$. Then for all $h,k\in\Z^d$ and all $m,n\in\N$ such that
$m\geq n-2$,
\begin{multline*}
\left|\sqrt{\psi_m(k-h)\psi_n(k)}-\sqrt{\psi_m(k)\psi_n(k-h)}\right|\\
=\left|\sqrt{\psi_m(k-h)\psi_n(k)}-\sqrt{\psi_m(k)\psi_n(k)}+\sqrt{\psi_m(k)\psi_n(k)}-\sqrt{\psi_m(k)\psi_n(k-h)}\right|\\
\leq L\frac{|h|}{2^m}\sqrt{\psi_n(k)}+L\frac{|h|}{2^n}\sqrt{\psi_m(k)}
\leq L\frac{|h|}{2^{n-3}}.
\end{multline*}
Moreover by simmetry with respect to $m$ and $n$,
\begin{multline*}
\sum_{k\in\Z^d}\left(\sqrt{\psi_m(k-h)\psi_n(k)}+\sqrt{\psi_m(k)\psi_n(k-h)}\right)\left|v_{k-h}\right| \left|v_{k}\right|\\
=2\sum_{k\in\Z^d}\sqrt{\psi_m(k-h)\psi_n(k)}\left|v_{k-h}\right| \left|v_{k}\right|,
\end{multline*}
so that
\[
|\phi_{(l,m,n)}|
\leq\frac{2^5L}{X_lX_mX_n}\sum_{h\in\Z^d\setminus\{0\}}|h|^{1-\beta}\psi_l(h)|v_h|\sum_{k\in\Z^d}\sqrt{\psi_m(k-h)\psi_n(k)}\left|v_{k-h}\right| \left|v_{k}\right|.
\]
By the usual bound $2^{l-1}\leq|h|\leq2^{l+1}$, since $\beta\geq0$, we
see that $|h|^{1-\beta}\leq2^{l(1-\beta)+1+\beta}$, so by
Lemma~\ref{l:cs_plus_cardinality_inequality},
\[
|\phi_{(l,m,n)}|
\leq 2^5 2^{(1-\beta)l+1+\beta} 2^{(l+3)\frac d2}L
\leq 2^{(\frac d2+1-\beta)(l-4)+9-3\beta+\frac{11}2d}L.
\]
Since in the present case $\min\{l,m,n\}\geq l-4$, this proves
inequality~\eqref{e:bound_phi} with $c_3=2^{9+\frac{11}2d-3\beta}L$.
\end{proof}

Finally we have all the ingredients to prove the main theorem of this section.

\begin{proof}[Proof of Theorem~\ref{t:shell_ode}]
A direct computation using~\eqref{e:xn_def}
and~\eqref{e:fourier_ns} shows that
\begin{multline*}
\frac12\frac d{dt}X_n^2
=\re\sum_{k\in\Z^d}\psi_n(k)\scalar{v'_k,v_k}\\
=-\sum_{k\in\zds}\psi_n(k)\frac{|k|^\alpha}{g(|k|)}|v_k|^2+\im\sum_{k\in\Z^d}\sum_{h\in\zds}\psi_n(k)\frac{\scalar{v_h,k}}{|h|^\beta}\scalar{P_k(v_{k-h}),v_k}\\
=-\sum_{k\in\zds}\psi_n(k)\frac{|k|^\alpha}{g(|k|)}|v_k|^2+\sum_{\substack{h,k\in\Z^d\\h\neq0}}\psi_n(k)\frac{\im\{\scalar{v_h,k}\scalar{v_{k-h},v_k}\}}{|h|^\beta}.
\end{multline*}
To deal with the first sum, define $\chi$ as in
Proposition~\ref{p:bound_chi}. By applying~\eqref{e:chi_def} for
$X_n(t)\neq0$ and~\eqref{e:xn_def} for $X_n(t)=0$ we see that in both
cases,
\[
2\sum_{k\in\zds}\psi_n(k)\frac{|k|^\alpha}{g(|k|)}|v_k|^2
=\chi_n(t)X_n^2(t).
\]
Now consider the second sum. Since the terms with $h=k$ give no
contribution, we can apply
\[
\sum_{l\in\natu}\psi_l(h)
=\sum_{m\in\natu}\psi_m(k-h)
=1
,\qquad h,k\in\Z^d,\quad 0\neq h\neq k,
\]
to get
\begin{multline*}
\sum_{\substack{h,k\in\Z^d\\h\neq0}}\psi_n(k)\frac{\im\{\scalar{v_h,k}\scalar{v_{k-h},v_k}\}}{|h|^\beta}\\
=\sum_{\substack{h,k\in\Z^d\\h\neq0}}\sum_{l,m\in\natu}\psi_l(h)\psi_m(k-h)\psi_n(k)\frac{\im\{\scalar{v_h,k}\scalar{v_{k-h},v_k}\}}{|h|^\beta}\\
=\sum_{l,m\in\natu}\sum_{\substack{h,k\in\Z^d\\h\neq0}}\psi_l(h)\psi_m(k-h)\psi_n(k)\frac{\im\{\scalar{v_h,k}\scalar{v_{k-h},v_k}\}}{|h|^\beta},
\end{multline*}
where it was possible to exchange the order of summation because the
middle expression is clearly absolutely convergent.

Now define $\phi$ as in Proposition~\ref{p:bound_phi}. By
applying~\eqref{e:phi_def} or~\eqref{e:xn_def} depending on
$X_l(t)X_m(t)X_n(t)$ being positive or zero, we see that for all
$l,m,n\in\natu$ and $t\geq0$,
\begin{multline*}
2\sum_{\substack{h,k\in\Z^d\\h\neq0}}\psi_l(h)\psi_m(k-h)\psi_n(k)\frac{\im\{\scalar{v_h,k}\scalar{v_{k-h},v_k}\}}{|h|^\beta}=\\
=\phi_{(l,m,n)}(t)X_l(t)X_m(t)X_n(t).
\end{multline*}
Putting all together we get
\[
\frac d{dt}X_n^2(t)
=-\chi_n(t)X_n^2(t)+\sum_{l,m\in\natu}\phi_{(l,m,n)}(t)X_l(t)X_m(t)X_n(t)
\qquad n\in\natu,\ t\geq0.
\]
Finally recalling by Proposition~\ref{p:bound_phi} that $\phi\equiv0$
outside $I$, we may restrict the scope of the sum and obtain
equation~\eqref{e:shell_ns}. The required properties of the
coefficients $\chi$ and $\psi$ follow again from
Propositions~\ref{p:bound_chi}-\ref{p:bound_phi}.
\end{proof}
\section{From the dyadic equation to the recursive inequality}
  \label{s:recursion}

In view of the results of the previous section, we can now concentrate
on shell solutions and forget equation \eqref{e:glanse}. In this section
we proceed as in \cite{BarMorRom2014} and we deduce a recursive inequality
between the tails of energy and dissipation. Clearly here, due to the
more complex non--linear interaction, the relation is less trivial
than in \cite{BarMorRom2014}.

\begin{definition}\label{d:energyineq}
A shell solution $X$ satisfies the \emph{energy inequality} on $[0,T]$
if the sum $\sum_n X_n^2(0)$ is finite and 
\begin{equation}\label{e:energy_inequality}
\sum_{n\in\natu}X_n^2(t)+\int_0^t\sum_{n\in\natu}\chi_n(s)X_n^2(s)ds
\leq\sum_{n\in\natu}X_n^2(0),
\qquad t\in[0,T].
\end{equation}
\end{definition}
\begin{definition}\label{d:def_df}
Let $X$ be a shell solution
and define the sequences of real valued maps $(F_n)_{n\in\natu}$ and
$(d_n)_{n\in\natu}$ for $t\geq0$ by
\begin{gather*}
F_n(t)
:=\sum_{k\geq n}X_k^2(t),\\\label{e:def_dn}
d_n(t)
:=\left(F_n(t)+\sum_{h\geq n}\int_0^t\chi_h(s)X_h^2(s)ds\right)^{\frac12}.
\end{gather*}
We will call $(F_n)_{n\in\natu}$ the \emph{tail} of $X$ and
$(d_n)_{n\in\natu}$ the \emph{energy bound} of $X$.
\end{definition}

The recursive inequality between the tails and the energy bound is
given in the next result.
\begin{proposition}\label{p:d_recursion}
  Let $X$ be a shell solution that satisfies the energy inequality on
  a time interval $[0,t]$, let $(d_n)_{n\in\natu}$ be its
  sequence of energy bounds, and set $\lambda=2^\alpha$.
  
  Then there is a positive constant $c_4>0$ such that for all $n\in\natu$,
\begin{equation}\label{e:d_recursion}
d_n^2(t)
\leq F_n(0)+c_4\sum_{l=0}^{n-1}\frac{\bar d_l}{\lambda^{n-l}}\sum_{m\geq n-2}\frac{g(2^{m+1})}{\lambda^{m-n}}\bigl(d_{m}^2(t)-d_{m+1}^2(t)\bigr),
\end{equation}
where $\bar d_l:=\max_{s\in[0,t]}d_l(s)$.
\end{proposition}
\begin{proof}
Fix $n\in\natu$. Differentiate $\sum_{h=0}^{n-1}X_h^2$ using
equation~\eqref{e:shell_ns},
\[
\frac d{dt}\sum_{h=0}^{n-1}X_h^2
=-\sum_{h=0}^{n-1}\chi_hX_h^2+\sum_{\substack{l,m,h\in\natu\\(l,m,h)\in I\\h\leq n-1}}\phi_{(l,m,h)}X_lX_mX_h.
\]
Apply Lemma~\ref{l:cancellations_phi} below to the second sum and integrate
on $[0,t]$ to obtain
\[
\sum_{h=0}^{n-1}X_h^2(t)-\sum_{h=0}^{n-1}X_h^2(0)
=-\int_0^t\sum_{h=0}^{n-1}\chi_hX_h^2\,ds-\int_0^t\sum_{\substack{(l,m,h)\in I\\m<n\leq h}}\phi_{(l,m,h)}X_lX_mX_h\,ds,
\]
so that by the energy inequality~\eqref{e:energy_inequality},
\[
F_n(t)+\int_0^t\sum_{h\geq n}\chi_h(s)X_h^2(s)\,ds
\leq F_n(0)+\int_0^t\sum_{\substack{(l,m,h)\in I\\m<n\leq h}}\phi_{(l,m,h)}X_l(s)X_m(s)X_h(s)\,ds,
\]
where the $F_n$ are the tails of $X$ and $F_n(0)<\infty$ by
hypothesis. Thus by \eqref{e:def_dn},
\[
d_n^2(t)
\leq F_n(0)+\int_0^t\sum_{\substack{(l,m,h)\in I\\m<n\leq h}}\phi_{(l,m,h)}X_l(s)X_m(s)X_h(s)\,ds.
\]
Recall that $\alpha+\beta\geq\frac{d}2+1$, hence
the bound \eqref{e:bound_chi_and_phi_def_shell_sol}
for $\phi$ yields $\phi_{(l,m,h)}
\leq c_2 \lambda^{\min\{l,m,h\}}$. Therefore
\[
d_n^2(t)
\leq F_n(0)+\int_0^t\sum_{\substack{(l,m,h)\in I\\m<n\leq h}}c_2\lambda^{\min\{l,m\}}|X_l(s)X_m(s)X_h(s)|\,ds.
\]
It is convenient to split the
set over which the sum is done into $\{l<m\}$ and $\{m\leq l\}$,
\begin{multline*}
\sum_{\substack{(l,m,h)\in I\\m<n\leq h}}\lambda^{\min\{l,m\}}|X_lX_mX_h|
\leq\sum_{\substack{(l,m,h)\in I\\l<m<n\leq h}}\lambda^l|X_lX_mX_h|
+\sum_{\substack{(l,m,h)\in I\\m<n\leq h\\m\leq l}}\lambda^m|X_lX_mX_h|\\
\leq\sum_{\substack{(l,m,h)\in I\\l<m<n\leq h}}\lambda^l|X_lX_mX_h|
+\sum_{\substack{(l,m,h)\in I\\l<n\leq h\\l\leq m}}\lambda^l|X_lX_mX_h|\\
\leq 2\sum_{\substack{(l,m,h)\in I\\l<n\leq h\\l\leq m}}\lambda^l|X_lX_mX_h|
\leq 2\sum_{l=0}^{n-1}\lambda^l\bar d_l\sum_{h\geq n}\sum_{m=h-2}^{h+2}|X_mX_h|.
\end{multline*}
Apply the Cauchy-Schwarz inequality to get
\[
2\sum_{h\geq n}\sum_{m=h-2}^{h+2}|X_hX_m|
\leq\sum_{h\geq n}\sum_{m=h-2}^{h+2}(X_h^2+X_m^2)
\leq10\sum_{m\geq n-2}X_m^2.
\]
Then by the bound on $\chi$ in~\eqref{e:bound_chi_and_phi_def_shell_sol},
on all $[0,t]$,
\[
\sum_{m\geq n-2}X_m^2
\leq c_1^{-1}\sum_{m\geq n-2}\frac{g(2^{m+1})}{\lambda^m}\chi_mX_m^2.
\]
Finally the integral of $\chi_mX_m^2$ can be bounded using \eqref{e:def_dn},
\[
d_{m}^2(t)-d_{m+1}^2(t)
=F_{m}(t)-F_{m+1}(t)+\int_0^t\chi_m(s)X_m^2(s)\,ds
\geq\int_0^t\chi_m(s)X_m^2(s)\,ds.
\]
Putting all together we obtain 
\[
d_n^2(t)
\leq F_n(0)+10\frac{c_2}{c_1}\sum_{l=0}^{n-1}\frac{\bar d_l}{\lambda^{-l}}\sum_{m\geq n-2}\frac{g(2^{m+1})}{\lambda^{m}}(d_{m}^2(t)-d_{m+1}^2(t)),
\]
thus proving equation~\eqref{e:d_recursion} with $c_4=10\frac{c_2}{c_1}$.
\end{proof}

\begin{lemma}\label{l:cancellations_phi}
Let $X$ be a shell solution, then for all $n\in\natu\setminus\{0\}$
and $s\in[0,t]$,
\begin{equation}\label{e:cancellations_phi}
\sum_{\substack{(l,m,h)\in I\\h\leq n-1}}\phi_{(l,m,h)}X_lX_mX_h
=-\sum_{\substack{(l,m,h)\in I\\m\leq n-1<h}}\phi_{(l,m,h)}X_lX_mX_h.
\end{equation}
\end{lemma}

\begin{proof}
By using \eqref{e:bound_chi_and_phi_def_shell_sol},
noticing that $\min(l,m,h)\leq n-1$, we see that
by definition of shell solution (Definition~\ref{d:shell_solution})
the left--hand side of \eqref{e:cancellations_phi} is an
absolutely convergent sum. Therefore we can exploit the
cancellations due to the antisymmetry of $\phi$, as in
Remark~\ref{r:cancellations_phi}. Indeed
\begin{equation}\label{e:long1}
  \smashoperator[r]{\sum_{\substack{(l,m,h)\in I\\h\leq n-1}}}\phi_{(l,m,h)}X_lX_mX_h
    = \smashoperator[r]{\sum_{\substack{(l,m,h)\in I\\m<h\leq n-1}}}\phi_{(l,m,h)}X_lX_mX_h
      +\smashoperator[r]{\sum_{\substack{(l,m,h)\in I\\h\leq n-1\\m>h}}}\phi_{(l,m,h)}X_lX_mX_h,
\end{equation}
and
\begin{multline}\label{e:long2}
  \sum_{\substack{(l,m,h)\in I\\h\leq n-1\\m>h}}\phi_{(l,m,h)}X_lX_mX_h
    = -\sum_{\substack{(l,m,h)\in I\\h\leq n-1\\m>h}}\phi_{(l,h,m)}X_lX_mX_h =\\
    = -\sum_{\substack{(l,h',m')\in I\\m'\leq n-1\\h'>m'}}\phi_{(l,m',h')}X_lX_{m'}X_{h'}
    = -\sum_{\substack{(l,m',h')\in I\\m'\leq n-1\\m'<h'}}\phi_{(l,m',h')}X_lX_{m'}X_{h'}.
\end{multline}
By using \eqref{e:long2} into \eqref{e:long1} the conclusion follows.
\end{proof}
\section{Solving the recursion}
  \label{s:cascade}

In this section we complete the proof of our main result. In the previous
section we have shown a recursive inequality involving the energy
bounds of a shell solution. The following theorem shows that
shell solutions are smooth. By Theorem~\ref{t:shell_ode}
the shell approximation of a solution of \eqref{e:glanse}
is a shell solution, hence Theorem~\ref{t:main2bis}
holds, and in turns Theorem~\ref{t:main2} holds as well.
\begin{theorem}\label{t:main2ter}
  Let $X$ be a shell solution
  satisfying the energy inequality on $[0,t)$.
  If $\sup_n 2^{mn}|X_n(0)|<\infty$ for every
  $m\geq1$, then
  \[
    \sup_{s\in [0,t]}\sup_n 2^{mn}|X_n(s)|
      <\infty
  \]
  for every $m\geq1$.
\end{theorem}
Let $b_n = g(2^{n+1})^{-1}$, $n\geq0$, then the assumptions of
Theorem~\ref{t:main2} for $g$ read in terms of the sequence $b$ as
\begin{itemize}
  \item $(b_n)_{n\in\N}$ non--increasing,
  \item $(\lambda^n b_n)_{n\in\N}$ non--decreasing,
  \item $\sum_n b_n = \infty$.
\end{itemize}
Let $X$ be a shell solution as in the statement of Theorem~\ref{t:main2ter},
denote by $(d_n)_{n\in\N}$ and $(F_n)_{n\in\N}$ the energy bound
and the tail of $X$ (see Definition~\ref{d:def_df}), and set
$\bar d_n = \sup_{[0,t]}d_n(t)$ for every $n$. Set
\[
  Q_n
    = \sum_{j=0}^{n-1}\frac{\bar d_j}{\lambda^{n-j}}
\]
and
\[
  R_n(t)
    = \sum_{j\geq n} \frac{d_j(t)^2 - d_{j+1}(t)^2}{\lambda^{j-n}b_j},
\]
where $\lambda=2^\alpha$ as in the previous section.
We recall that, by Proposition~\ref{p:d_recursion}, the
following inequality holds,
\begin{equation}\label{e:drecursive}
  d_n(t)^2
    \leq F_n(0) + c_4Q_n R_{n-2}(t).
\end{equation}
In the following lemma we collect some properties of the quantities
$R_n$, $Q_n$, $\bar d_n$ that will be crucial in the proof of
Theorem~\ref{t:main2ter} above.
\begin{lemma}
  The following properties hold.
  \begin{enumerate}
    \item For every $1\leq m_1\leq m_2$ and $t>0$,
      \begin{equation}\label{e:Rformula}
        \min\{R_{m_1}(t),R_{m_1+1}(t)\dots,R_{m_2}(t)\}
          \leq \frac{\lambda}{\lambda-1}
            \frac{d_{m_1}(t)^2}{\sum_{n=m_1}^{m_2} b_n}.
      \end{equation}
    \item For every $t>0$, $\liminf_n R_n(t) = 0$.
    \item $\bar d_n\downarrow0$ as $n\to\infty$.
    \item $Q_n\to0$ as $n\to\infty$.
    \item $(Q_n)_{n\geq1}$ is eventually non--increasing.
  \end{enumerate}
\end{lemma}
\begin{proof}
  Since $\lambda^n b_n$ is non--decreasing, we know that
  $b_n - \lambda^{-1}b_{n-1}\geq0$. Hence by exchanging the sums,
  \begin{multline*}
      \sum_{n=m_1}^\infty\bigl(b_n - \lambda^{-1}b_{n-1}\bigr)R_n(t)=\\
        = \sum_{k=m_1}^\infty \frac{d_k(t)^2 - d_{k+1}(t)^2}{\lambda^k b_k}
          \sum_{n=m_1}^k \bigl(\lambda^n b_n - \lambda^{n-1}b_{n-1}\bigr)\leq\\
        \leq \sum_{k=m_1}^\infty (d_k(t)^2 - d_{k+1}(t)^2)
        \leq d_{m_1}(t)^2.
  \end{multline*}
  If $m_2\geq m_1$, since $(b_n)_{n\geq1}$ is non--increasing,
  \[
    \begin{aligned}
      \sum_{n=m_1}^{m_2} \bigl(b_n - \lambda^{-1}b_{n-1}\bigr)R_n(t)
        &\geq \min\{R_{m_1}(t),\dots,R_{m_2}(t)\}
          \sum_{n=m_1}^{m_2}\bigl(b_n - \lambda^{-1}b_{n-1}\bigr)\\
        &\geq\frac{\lambda-1}{\lambda}\Bigl(\sum_{n=m_1}^{m_2} b_n\Bigr)
          \min\{R_{m_1}(t),\dots,R_{m_2}(t)\}.
    \end{aligned}
  \]
  
  The claim $\liminf_n R_n(t)=0$ follows from \eqref{e:Rformula}, since
  $d_n(t)\leq d_1(t)$ for every $n$, and since, by the assumptions on
  $(b_n)_{n\geq1}$, we can find a sequence $(m_k)_{k\geq1}$ such that
  $\sum_{n=m_k}^{m_{k+1}-1}b_n\uparrow\infty$.
  
  To prove that $\bar d_n\downarrow0$, we notice that the sequence
  $(m_k)_{k\geq1}$ mentioned above does not depend on $t$, hence
  using the monotonicity of $(d_n(t))_{n\geq1}$ and formula
  \eqref{e:Rformula}, we can prove that $\liminf_n \bar d_n = 0$,
  and hence $\bar d_n\downarrow0$ by monotonicity.
  Once we know that $\bar d_n\downarrow0$, an easy and standard
  argument proves that $Q_n\to0$.
   
  To prove that $(Q_n)_{n\geq1}$ is eventually non--increasing, we
  notice that, since $(\bar d_n)_{n\geq1}$ is non--increasing,
  \[
    (Q_{n+1} - Q_n)
      = \frac1\lambda(Q_n - Q_{n-1}) + \frac1\lambda(\bar d_n - \bar d_{n-1})
      \leq \frac1\lambda(Q_n - Q_{n-1}).
  \]
  In view of the above inequality, it is sufficient to show that for some $m$ the
  increment $Q_m-Q_{m-1}\leq0$. This is true because otherwise the sequence
  $(Q_n)_{n\geq1}$ would be non--decreasing, in contradiction with $Q_n\to0$
  and $Q_n\geq0$.
\end{proof}
Given $\theta>0$ and $n_0\geq1$, define by recursion the
sequence
\begin{equation}\label{e:sequence}
  n_{k+1}
    = 2 + \min\Bigl\{n\geq n_k-1: \sum_{j=n_k-1}^n b_j\geq\theta\lambda^{-\frac{k}4}\Bigr\}.
\end{equation}
The definition of $Q_n$ and the fact that the sequence $(\bar d_n)_{n\geq1}$ is
non--increasing yield the following recursive formula for $Q_{n_k}$,
\begin{equation}\label{e:Qrecursive}
   Q_{n_{k+1}}
     = \frac1{\lambda^{n_{k+1}-n_k}}Q_{n_k}
       + \sum_{j=n_k}^{n_{k+1}-1}\frac{\bar d_j}{\lambda^{n_{k+1}-j}}
     \leq \frac1\lambda Q_{n_k} + c\bar d_{n_k},
\end{equation}
for a constant $c>0$ depending only from $\lambda$. Moreover,
if we choose $n_0$ large enough that $(Q_n)_{n\geq0}$ is
non--increasing,
\[
  d_{n_{k+1}}(t)^2
    \leq d_n(t)^2
    \leq F_n(0) + c_4 Q_n R_{n-2}(t)
    \leq F_{n_k}(0) + c_4 Q_{n_k} R_{n-2}(t)
\]
for each $n\in\{n_k+1,\dots,n_{k+1}\}$, hence
by formula \eqref{e:Rformula} and the definition of
the sequence $(n_k)_{k\geq1}$,
\begin{multline*}
  d_{n_{k+1}}(t)^2
    \leq F_{n_k}(0) + c_4 Q_{n_k}\min\{R_{n_k-1},\dots,R_{n_{k+1}-2}\}\leq\\
    \leq F_{n_k}(0) + c Q_{n_k}\frac{d_{n_k-1}(t)^2}{\sum_{n_k-1}^{n_{k+1}-2}b_j}
    \leq F_{n_k}(0) + c\frac{\lambda^{\frac{k}4}}{\theta}Q_{n_k}d_{n_k-1}(t)^2,
\end{multline*}
and in conclusion,
\begin{equation}\label{e:bdrecursive}
  \bar d_{n_{k+1}}^2
    \leq F_{n_k}(0)
      + c\frac{\lambda^{\frac{k}4}}{\theta}Q_{n_k}\bar d_{n_k-1}^2.
\end{equation}
\begin{lemma}[initial step of the cascade]\label{l:initial}
  Given $M>0$, there are $n_0\geq1$ and $\theta>0$ such that 
  \[
    \begin{gathered}
      Q_{n_k}
        \leq \lambda^{-\frac{k}2},\\
      \bar d_{n_k}^2
        \leq \lambda^{-Mk},
    \end{gathered}
  \]
  for all $k\geq0$.
\end{lemma}
\begin{proof}
  Without loss of generality we can choose $M$ large (depending only
  on the value of $\lambda$, see below at the end of the proof).
  Choose $n_0$ large enough that $(Q_n)_{n\geq n_0}$ is
  non--increasing and
  \[
    Q_{n_0-i}
      \leq\epsilon,
        \qquad
    \bar d_{n_0-i}
      \leq \epsilon,
        \quad
    i=0,1,
        \quad\text{and}\quad
    \lambda^{Mn}F_n(0)
      \leq\epsilon,
        \quad n\geq n_0,
  \]
  for a number $\epsilon\in(0,1)$ suitably chosen below.
  We will prove by induction that
  \begin{equation}\label{e:initial_claim}
    Q_{n_k-i}
      \leq \lambda^{-\frac12(k-i)},
        \qquad
    \bar d_{n_{k-i}}^2
      \leq \lambda^{-M(k-i)},
        \qquad
    i=0,1,
        \qquad
    k\geq1.
  \end{equation}
  For the initial step of the induction ($k=1$), we notice that
  by \eqref{e:Qrecursive} and \eqref{e:bdrecursive},
  \[
    \begin{gathered}
      Q_{n_1}
        \leq \frac1\lambda Q_{n_0}
          + c\bar d_{n_0}
        \leq \frac\epsilon\lambda + c\epsilon
        \leq \frac1{\lambda^{1/2}},\\
      \bar d_{n_1}^2
        \leq F_{n_0}(0) + \frac{c}\theta Q_{n_0}\bar d_{n_0-1}^2
        \leq \epsilon + \frac{c}{\theta}\epsilon^3
        \leq \lambda^{-M},\\
    \end{gathered}
  \]
  if we choose $\epsilon$ small enough, depending
  from the values of $\lambda$, $M$, and $\theta$.

  Assume now that \eqref{e:initial_claim} holds for some $k\geq1$,
  and let us prove that the same holds for $k+1$. To this end
  it is sufficient to give the estimate for $Q_{n_{k+1}}$
  and $\bar d_{n_{k+1}}^2$.
  Again by \eqref{e:Qrecursive}, \eqref{e:bdrecursive}
  and the induction hypothesis, and since $(n_k)_{k\geq0}$
  is increasing by definition,
  \[
    \begin{gathered}
      Q_{n_{k+1}}
        \leq \frac1\lambda Q_{n_k}
          + c \bar d_{n_k}
        \leq \lambda^{-\frac{k}2-1} + c\lambda^{-\frac{M}2k}
        \leq \lambda^{-\frac12(k+1)},\\
      \bar d_{n_{k+1}}^2
        \leq F_{n_k}(0)
          + c\frac{\lambda^{\frac{k}4}}{\theta}Q_{n_k}\bar d_{n_k-1}^2
        \leq \epsilon\lambda^{-Mk}
          + \frac{c}\theta \lambda^{-\frac{k}4}\lambda^{-M(k-1)}
        \leq \lambda^{-M(k+1)},
    \end{gathered}
  \]
  if $M$ is large (depending on $\lambda$), and $\epsilon$ is small
  and $\theta$ is large (depending only on $M$, $\lambda$).
\end{proof}
Before giving the last step of the proof of Theorem~\ref{t:main2ter},
we show a property of the sequence $(n_k)_{k\geq0}$. The proof is
the same as \cite[Lemma 11]{BarMorRom2014}, we detail it for
completeness.
\begin{lemma}\label{l:indices}
  Given $n_0\geq1$ and $\theta>0$, consider the sequence
  defined in \eqref{e:sequence}.
  For infinitely many $k$, $n_{k+1}=n_k+1$. In particular
  $b_{n_k-1}\geq\theta\lambda^{-k/4}$ for all such $k$.
\end{lemma}
\begin{proof}
  Assume by contradiction that
  there is $r$ such that $n_{k+1}\geq n_k+2$ for $k\geq r$.
  On the one hand
  \[
    \sum_{j=n_k-1}^{n_{k+1}-3}b_j
      \leq \theta \lambda^{-\frac{k}4},
  \]
  and summing up in $k\geq r$ yields
  \[
    \sum_{k\geq r}\sum_{j=n_k-1}^{n_{k+1}-3}b_j
      <\infty
        \qquad\leadsto\qquad
    \sum_k b_{n_k-2}
      = \infty.
  \]
  On the other hand, $b_{n_k-2}\leq b_{n_k-3}\leq
  \theta\lambda^{-\frac14(k-1)}$ and the series
  $\sum_k b_{n_k-2}$ converges.
\end{proof}
\begin{lemma}[cascade recursion]
  For every $M>0$ there is $c_M>0$ such that
  \[
    \bar d_n^2
      \leq c_M \lambda^{-Mn},
        \qquad\quad
    Q_n
      \leq c_M \lambda^{-n}.
  \]
\end{lemma}
\begin{proof}
  There is no loss of generality if we assume $M$ is large.
  Let $n_0$, $\theta$ be the values
  provided by Lemma~\ref{l:initial}.
  By Lemma~\ref{l:initial} and Lemma~\ref{l:indices}
  there are infinitely many $k\geq1$ such that
  \begin{equation}\label{e:cascade_init}
    b_{n_k-1}
      \geq\theta \lambda^{-\frac{k}4},
        \qquad
    Q_{n_k}
      \leq \lambda^{-\frac{k}2},
        \qquad
    \bar d_{n_k}^2
      \leq \lambda^{-Mk}.
  \end{equation}
  Let $k_0$ be one of such indices, large enough (the size of $k_0$
  will be chosen at the end of the proof). We will prove
  by induction that
  \begin{equation}\label{e:cascade_recurse}
    \bar d_{n_{k_0}+m}^2
      \leq c \lambda^{-Mm},
        \qquad
    Q_{n_{k_0}+m}
      \leq c' \lambda^{-m},
        \qquad
    b_{n_{k_0}-1+m}
      \geq \theta \lambda^{-\frac{k_0}4-m},
  \end{equation}
  for a suitable choice of the constants $c>0$, $c'>0$. We first
  notice that there is nothing to prove concerning $b_{n_{k_0}-1+m}$,
  since this is a straightforward consequence of the choice of $k_0$
  and the monotonicity of $(\lambda^n b_n)_{n\geq1}$.
  
  The initial step $m=0$ holds, since inequalities \eqref{e:cascade_init}
  hold true for the index $k_0$. For $m=1$,
  \[
    \begin{gathered}
      \bar d_{n_{k_0}+1}^2
        \leq \bar d_{n_{k_0}}^2
        \leq c \lambda^{-M},\\
      Q_{n_{k_0}+1}
        = \frac1\lambda Q_{n_{k_0}} + \frac1\lambda\bar d_{n_{k_0}}
        \leq \frac1\lambda(\lambda^{-\frac{k_0}2} + \lambda^{-\frac{M}2k_0})
        \leq \frac{c'}\lambda,
    \end{gathered}
  \]
  if $c=\lambda^{-M(k_0-1)}$ and $c'\geq\lambda^{-k_0/2} + \lambda^{-Mk_0/2}$.
  
  Assume that \eqref{e:cascade_recurse} holds for $1,\dots,m$,
  for some $m\geq1$. By its definition,
  \[
    \begin{aligned}
      Q_{n_{k_0}+m+1}
        &= Q_{n_{k_0}}\lambda^{-(m+1)}
          + \sum_{j=n_{k_0}}^{n_{k_0}+m}\frac{\bar d_j}{\lambda^{n_{k_0}+m+1-j}}\\
        &\leq \lambda^{-\frac{k_0}2-(m+1)}
          + \sqrt{c}\lambda^{-(m+1)}\sum_{j=0}^m\lambda^{-(\frac{M}2-1)j}\\
        &\leq \Bigl(\lambda^{-\frac{k_0}2} + \frac\lambda{\lambda-1}\sqrt{c}\Bigr)
          \lambda^{-(m+1)},\\
        &\leq c' \lambda^{-(m+1)},
    \end{aligned}
  \]
  if $c'= \lambda^{-\frac{k_0}2} + \lambda(\lambda-1)^{-1}\sqrt{c}$
  (the previous constraint on $c'$ is met by this choice).
 
  By \eqref{e:drecursive} and \eqref{e:Rformula} we have that for
  every $n\geq2$,
  \[
    d_{n+1}(t)^2
      \leq F_{n+1}(0) + 04 Q_{n+1} R_{n-1}(t)
      \leq F_{n+1}(0) + c_4 Q_{n+1}\frac{\bar d_{n-1}^2}{b_{n-1}},
  \]
  hence, using the inequality for $Q_{n_{k_0}+m+1}$ already
  proved and the induction hypothesis,
  \[
    \begin{aligned}
      \bar d_{n_{k_0}+m+1}^2
        &\leq F_{n_{k_0}+m+1}(0)
           + c_4 Q_{n_{k_0+m+1}}\frac{\bar d_{n_{k_0}+m-1}^2}{b_{n_{k_0}+m-1}}\\
        &\leq c\lambda^{-M(m+1)}\Bigl(
           \lambda^{M(n_{k_0}+m+1)}F_{n_{k_0}+m+1}(0)
           + \frac{c_4}\theta c'\lambda^{2M+\frac{k_0}4}\Bigr)\\
        &\leq c 2^{-M(m+1)},
    \end{aligned}
  \]
  where the last inequality follows if $k_0$ is large enough, since
  $\lambda^n F_n(0)\to0$ by assumption, and by our choice of $c$, $c'$,
  we have that $\lambda^{k_0/4}c'\to0$
  as $k_0\to\infty$.
\end{proof}
\appendix
\section{Local existence and uniqueness}
  \label{s:localexuniq}

Consider the generalised system \eqref{e:glanse}, under the same
assumptions of Theorem~\ref{t:main2}. Assume\footnote{Existence
    and uniqueness can be proved also in the
    general case $m_1(k)\geq|k|^\alpha g(|k|)^{-1}$.
    A simple assumption that keeps our proof almost unchanged
    is a control from above, say $m(k)\leq |k|^\beta$, for some
    $\beta\geq\alpha$.},
for simplicity, that $m_1(k)=\frac{|k|^\alpha}{g(|k|)}$.
Denote by $V_m$ the subspace of $H^m$ (see \eqref{e:sobolev})
of divergence free vector fields with mean zero. 
Our main theorem on local existence and uniqueness for
\eqref{e:glanse} is as follows.
\begin{theorem}\label{t:localexuniq}
  Let $m\geq 2+\frac{d}2$ and $v_0\in V_m$. Then there
  are $T>0$ and a unique solution $v$ of \eqref{e:glanse}
  on $[0,T]$ with initial condition $v_0$ such that
  \begin{equation}\label{e:leu_reg}
    \begin{gathered}
      v\in L^\infty([0,T];V_m)\cap\textup{Lip}([0,T];V_{m-\alpha})
        \cap C([0,T];V_m^\textup{\tiny weak}),\\
      \int_0^T \|D_1^{\frac12} v\|_m^2\,dt
        <\infty,
    \end{gathered}
  \end{equation}
  where $V_m^\textup{\tiny weak}$ is the space $V_m$ with the
  weak topology. Moreover, $v$ is right--continuous with
  values in $V_m$ for the strong topology.
  
  If $T_\star$ is the maximal time of existence of the solution
  started from $v_0$, then either $T_\star=\infty$ or
  \[
    \limsup_{t\uparrow T_\star}\|v(t)\|_m
      = \infty.
  \]
\end{theorem}
The proof of the theorem is based on a proof of existence of
a local unique solution for the Euler equation taken from
\cite[Section 3.2]{BerMaj2002}. The idea is that we cannot
use the $D_1$ operator as a replacement for the Laplacian,
since in general $D_1$ may not have smoothing properties
(indeed, it is easy to adapt the counterexample in
\cite[Remark 15]{BarMorRom2014} to $D_1$ on $\R^d$ or
on the $d$--dimensional torus). Likewise we do not use any smoothing
properties of $D_2$, so that our proof includes
the case $\beta=0$. The result is by no means optimal,
but fits the needs of our paper.

We work on the torus $[0,2\pi]^d$, although
the proof, essentially unchanged, works in $\R^d$.
Denote by $H$ the projection of $L^2([0,2\pi]^d)$
onto divergence free vector fields,
and for every $s>0$, by $V_s$
the projection of the Sobolev space $H^s([0,2\pi]^d)$
onto divergence free vector fields.
We will denote by $\|\cdot\|_H$ and by
$\scalar{\cdot,\cdot}_H$ the norm and the scalar product
in $H$, and by $\|\cdot\|_s$ and by
$\scalar{\cdot,\cdot}_s$ the norm and the scalar product
in $V_s$.

We denote by $\hat B(v_1,v_2)$ the (Leray) projection of the
non--linearity, namely
\[
  \hat B(v_1,v_2)
    = \Pi_\text{Leray}\bigl[\bigl((D_2^{-1}v_1\cdot\nabla\bigr) v_2\bigr].
\]
Since $\beta\geq0$, $\|D_2^{-1}v\|_s\leq\|v\|_s$ for
every $s\in\R$. Hence, (see for instance \cite{Kat1972},
or \cite{ConFoi1988}), for every $m\geq1+[\tfrac{d}2]$,
there exists $c_m>0$ such that
\[
  \begin{gathered}
    \|\hat B(v_1,v_2)\|_m	
      \leq c_m\|v_1\|_m\|v_2\|_{m+1},\\
    \scalar{\hat B(v_1,v_2), v_2}_m
      \leq c_m \|v_1\|_m\|v_2\|_m^2.
  \end{gathered}
\]

In the rest of the section we briefly outline the proof of
Theorem~\ref{t:localexuniq}, following \cite[Section 3.2]{BerMaj2002}.
The proof of the following result is a slight modification
of the arguments to prove \cite[Theorem 3.4]{BerMaj2002}.
\begin{proposition}\label{p:localex}
  Given an integer
  $m\geq 2 + \tfrac{d}2$, there exists a number
  $c_\star>0$ such that for every $v_0\in V_m$,
  if $T<c_\star/\|v_0\|_m$, there is a unique solution
  of \eqref{e:glanse} with initial condition $v_0$.
  Moreover $\vep\to v$ in $C([0,T];V_{m'})$, for $m'<m$,
  and in $C([0,T];V_m^\textup{\tiny weak})$,
  \eqref{e:leu_reg} hold for $v$, and
  for every $\epsilon>0$,
  \begin{equation}\label{e:bound}
    \sup_{[0,T]}\|\vep\|_m
      \leq \frac{\|v_0\|_m}{1 - c_\star T\|v_0\|_m}.
  \end{equation}
\end{proposition}
Unfortunately, at this stage, we cannot prove the analogous of
Theorem 3.5 of \cite{BerMaj2002} for our $v$, namely that
$v$ is continuous in time for the strong topology of $V_m$.
The reason is that their proof uses either the reversibility
of the Euler equation (that we do not have due to the presence
of $D_1$), or the smoothing of the Laplace operator, that
we do not have here either (as already mentioned). On
the other hand we can prove right--continuity.
\begin{lemma}
  The solution $v$ from Proposition~\ref{p:localex} is
  right--continuous with values in $V_m$ for the
  strong topology, and $\frac{d}{dt}v$ is right
  continuous with values in $V_{m-\alpha}$.
\end{lemma}
\begin{proof}
  Given $t\in [0,T]$, the same computations leading
  to \eqref{e:bound} yield
  \[
    \sup_{[0,t]}\|v(s)\|_m
      \leq \|v_0\|_m + \frac{c_\star t\|v_0\|_m^2}{1 - c_\star t\|v_0\|_m},
  \]
  therefore $\limsup_{t\downarrow0}\|v(t)\|_m\leq \|v_0\|_m$.
  On the other hand, by weak continuity, $\|v_0\|_m\leq
  \liminf_{t\downarrow0}\|v(t)\|_m$ and $v$ is right
  continuous in $0$.
  Uniqueness for \eqref{e:glanse} and the same argument applied to
  $t\in(0,T]$ yield right--continuity in $t$.
\end{proof}
Nevertheless, we can still define a maximal solution and
a maximal time of existence. Given $v_0\in V_m$, let $T_\star$
be the maximal time of existence of the solution starting from $v_0$,
that is the supremum over all $T>0$ such that there exists
a solution $v$ of \eqref{e:glanse} on $[0,T]$ with $v(0)=u_0$,
$v$ is right--continuous with values in $V_m$, continuous
with values in $V_m^\textup{\tiny weak}$ and with
$\frac{d}{dt}v$ right continuous with values in $V_{m-\alpha}$.
Due to uniqueness, any two such solutions coincide on the
common interval of definition.
\begin{proposition}
  Given $v_0\in V_m$, if $T_\star$ is the maximal time of existence
  of the solution started from $v_0$, then either $T_\star=\infty$ or
  \[
    \limsup_{t\uparrow T_\star} \|v(t)\|_m
      = \infty.
  \]
\end{proposition}
\begin{proof}
  Assume by contradiction that $T_\star<\infty$ and that
  $M:=\sup_{t<T_\star}\|v(t)\|_m<\infty$. Let
  $T_0=T_\star-c_\star/(4M)$, and start a solution with
  initial condition $v(T_0)$ at time $T_0$. By
  Proposition~\ref{p:localex} there is a solution
  of \eqref{e:glanse} on a time span of length at least
  $c_\star/(2\|v(T_0)\|_m)\geq c_\star/(2M)$,
  hence at least up to time $T_0+c_\star/(2M)>T_\star$.
  By uniqueness, this solution is equal to $v$ up to
  time $T_\star$.
\end{proof}

\end{document}